\documentclass{aims}

\usepackage{amsmath}
\usepackage{paralist}
\usepackage[colorlinks=true]{hyperref}
 \hypersetup{urlcolor=blue, citecolor=red}

\def\currentvolume{X}
 \def\currentissue{X}
  \def\currentyear{200X}
   \def\currentmonth{XX}
    \def\ppages{X--XX}

\newtheorem{theorem}{Theorem}[section]
\newtheorem{lemma}[theorem]{Lemma}
\newtheorem{proposition}[theorem]{Proposition}
\theoremstyle{definition}
\newtheorem{definition}[theorem]{Definition}
\newtheorem{remark}[theorem]{Remark}

\title[Extensions of expansive dynamical systems]
      {Extensions of expansive dynamical systems}

\author[Mauricio Achigar]{}

\subjclass{Primary: 37B05, 37B65; Secondary: 37B02, 37B25.}

\keywords{Semi expansive system, Semi Anosov system, Lewowicz quotient, Topological stability, Shadowing envelopes.}

\email{machigar@unorte.edu.uy}

\thanks{Partially supported by Agencia Nacional de Investigación e Innovación, Uruguay.}

\usepackage{amsfonts}
\usepackage{amsrefs}
\usepackage{amssymb}
\usepackage{mathrsfs}
\usepackage{nicefrac}

\theoremstyle{definition}
\newtheorem{example}[theorem]{Example}

\DeclareMathOperator{\diam}{{\rm diam}}
\DeclareMathOperator{\im}{{\rm Im}}
\DeclareMathOperator{\homeo}{{\mathcal H}}
\DeclareMathOperator{\cont}{{\mathcal C}}

\newcommand{\Z}{\mathbb Z}
\newcommand{\N}{\mathbb N}
\newcommand{\U}{\mathscr U}
\newcommand{\V}{\mathscr V}
\newcommand{\C}{\mathscr C}
\newcommand{\D}{\mathscr D}
\newcommand{\Set}{\mathcal S}
\newcommand{\NN}{\mathcal N}
\renewcommand{\labelenumi}{\rm\arabic{enumi}) \!}

\makeatletter
 \newsavebox\myboxA
 \newsavebox\myboxB
 \newlength\mylenA
 \newcommand*\xoverline[2][0.75]{%
  \sbox{\myboxA}{$\m@th#2$}%
  \setbox\myboxB\null
  \ht\myboxB=\ht\myboxA%
  \dp\myboxB=\dp\myboxA%
  \wd\myboxB=#1\wd\myboxA
  \sbox\myboxB{$\m@th\overline{\copy\myboxB}$}
  \setlength\mylenA{\the\wd\myboxA}
  \addtolength\mylenA{-\the\wd\myboxB}%
  \ifdim\wd\myboxB<\wd\myboxA%
  \rlap{\hskip 0.5\mylenA\usebox\myboxB}{\usebox\myboxA}%
  \else
  \hskip -0.5\mylenA\rlap{\usebox\myboxA}{\hskip 0.5\mylenA\usebox\myboxB}%
  \fi}
\makeatother

\begin{document}

\begin{abstract}
 We characterize and describe the extensions of expansive and Ano- sov homeomorphisms on compact spaces. As an application we obtain a stability result for extensions of Anosov systems, and show a construction that embeds any expansive system inside an expansive system having the shadowing property for the pseudo orbits of the original space.
\end{abstract}

\maketitle

\centerline{\scshape Mauricio Achigar}

\medskip

{\footnotesize
 \centerline{Departamento de Matemática y Estadística del Litoral}
 \centerline{Centro Universitario Regional Litoral Norte}
 \centerline{Universidad de la República}
 \centerline{25 de Agosto 281, Salto (50000), Uruguay}}

\bigskip

 \centerline{(Communicated by the associate editor name)}

\section{Introduction}

Let $(M,d)$ be a compact metric space. A homeomorphism $f\colon M\to M$ is \emph{expansive} if there is $\alpha>0$ (an \emph{expansivity constant}) such that 
\begin{center}
 if $x,y\in M$ and $d(f^nx,f^ny)\leq\alpha$ for all $n\in\Z$\quad then\quad  $x=y$. 
\end{center}
In \cite{Le83}*{Lemma 1.1} Lewowicz observed that although we cannot assume that if $f_0\colon M\to M$ is an expansive homeomorphism a perturbed homeomorphism $f$ is expansive, it is true that if $\alpha>0$ is an expansivity constant for $f_0$, $0<\varepsilon<\alpha$, and $f$ is close enough to $f_0$ then $f$ has the following property resembling expansiveness
\begin{equation}\label{ecu:A}
 \text{if $x,y\in M$ and $d(f^nx,f^ny)\leq\alpha$ for all $n\in\Z$\quad  then\quad  $d(x,y)\leq\varepsilon$.}
\end{equation}
Taking $0<\varepsilon\leq\nicefrac\alpha2$ he introduced the equivalence relation on $M$ given by
\begin{equation}\label{ecu:B}
 \text{$x\sim y$\quad iff\quad $d(f^nx,f^ny)\leq\alpha$ for all $n\in\Z$,}
\end{equation}
and studied the properties of the quotient space $M/{\sim}$ and the quotient homeomorphism $\tilde f\colon M/{\sim}\to M/{\sim}$, suggesting that $\tilde f$ is expansive and showing that if $f_0$ is \emph{persistent} then there is a copy of $(M,f_0)$ inside $(M/{\sim},\tilde f)$. In a subsequent paper, \cite{CS}, Cerminara and Sambarino proved in full detail that $\tilde f$ is expansive, and for the case of $M$ a surface they showed that the local stable and unstable sets of the system $(M/{\sim},\tilde f)$ are non-trivial and arc-wise connected.

The aim of this paper is to study the previous situation assuming directly a property of the form (\ref{ecu:A}) for $f$ without supposing the existence of $f_0$ expansive; in fact the existence of such an $f_0$ now turns to be a question. In \S\ref{sec:alphaepsilonexp} we call this type of homeomorphisms \emph{$[\varepsilon,\alpha]$-expansive homeomorphisms} (see Definition \ref{def:alpha_epsilon_exp} for the precise meaning of this term which slightly differs from condition (\ref{ecu:A}) above), proving some properties of this class of systems such as the \emph{uniform $[\varepsilon,\alpha]$-expansiveness} property in Proposition \ref{prop:unif_alplha_epsilon_exp} and an openness result in Proposition \ref{prop:epsilon_alpha_exp_es_abierto}.

For the case $0<\varepsilon\leq\nicefrac\alpha2$ we obtain what we called \emph{semi expansive homeomorphisms} in \S\ref{sec:expansive_quotients}, where, in Theorem \ref{teo:cociente_expnsivo}, we show that for these systems the equivalence relation (\ref{ecu:B}) induces an expansive quotient $\tilde f\colon M/{\sim}\to M/{\sim}$, and conversely that any expansive quotient system of a compact metrizable system, case in which the latter is called an \emph{extension} of the former, arises in that way. In \S\ref{subsec:topolog_charact} we consider this problem from a purely topological viewpoint characterizing the compact and non necessarily metrizable extensions of expansive systems.

In \S\ref{sec:anosov_quotients} we go further establishing necessary and sufficient conditions for a compact metrizable system to have an \emph{Anosov} quotient (an expansive system with the \emph{shadowing property}, see Definition \ref{def:shadowing}), describing all such extensions in Theorem \ref{teo:cociente_anosov} by means of what we called \emph{semi Anosov homeomorphisms} (see Definition \ref{def:semi_anosov}). Finally, in \S\ref{subsec:estabilidad_gorda} and \S\ref{subsec:envovente} we give two applications: a stability result in Theorem \ref{teo:estabilidad}, extending the \emph{topological stability} property of Anosov systems proved by Walters in \cite{Wa78}*{Theorem 4}; and Theorem \ref{teo:envolvente} where we show, loosely speaking, that any expansive system can be enlarged to an expansive system satisfying the shadowing property for the pseudo orbits of the original system (see Definition \ref{def:envelope}).

\section*{Context} Throughout this paper unless otherwise stated $M$ will denote a compact metric space with metric $d$ and $f\colon M\to M$ a homeomorphism.

\section{On $[\varepsilon,\alpha]$-expansiveness}\label{sec:alphaepsilonexp}

In this section we introduce the class of $[\varepsilon,\alpha]$-expan- sive homeomorphisms and state some of their properties. Let us start recalling the definition of expansiveness.

\begin{definition}\label{def:expansivo}
 We say that $f$ is \emph{expansive} iff there exists $\alpha>0$ such that
 \begin{center}
  if $x,y\in M$ and $d(f^nx,f^ny)\leq\alpha$ for all $n\in\Z$ then $x=y$.
 \end{center}
 In that case $\alpha$ is called an \emph{expansivity constant} for $f$, and $f$ is said to be \emph{$\alpha$-expansive}.
\end{definition}

\begin{definition}\label{def:alpha_epsilon_exp}
 Given constants $0<\varepsilon\leq\alpha$ we say that $f$ is \emph{$[\varepsilon,\alpha]$-expansive} iff
 \begin{center}
  $x,y\in M$ and $d(f^nx,f^ny)\leq\alpha$ for all $n\in\Z$ implies $d(x,y)<\varepsilon$.
 \end{center}
 If $f$ is $[\varepsilon,\alpha]$-expansive and $\diam M<\varepsilon$ we say that $f$ is \emph{trivially} $[\varepsilon,\alpha]$-expansive.\footnote{ $\diam M=\max\{d(x,y):x,y\in M\}\in[0,+\infty)$ denotes the diameter of $M$.}
\end{definition}

From the definition we see that if $f$ is $[\varepsilon,\alpha]$-expansive and $d(f^nx,f^ny)\leq\alpha$ for all $n\in\Z$ then $d(f^nx,f^ny)<\varepsilon$ for all $n\in\Z$. It is also easily checked that $f$ is $\alpha$-expansive if and only if $f$ is $[\varepsilon,\alpha]$-expansive for all $\varepsilon>0$ less than $\alpha$.

\begin{lemma}\label{lem:gap_abierto}
 If $f$ is $[\varepsilon,\alpha]$-expansive, $f$ is $[\varepsilon-\delta,\alpha+\delta]$-expansive for some $\delta>0$.
\end{lemma}

\begin{proof}
 If the thesis were false then for all $k\in\N$, $k>\nicefrac1\varepsilon$, we have that $f$ is not $[\varepsilon-\nicefrac1k,\alpha+\nicefrac1k]$-expansive, that is, $d(f^nx_k,f^ny_k)\leq\alpha+\nicefrac1k$ for all $n\in\Z$ and $d(x_k,y_k)\geq\varepsilon-\nicefrac1k$ for suitable $x_k,y_k\in M$. As $M$ is compact there exist limit points $x$ and $y$ of the sequences $(x_k)$ and $(y_k)$, respectively, which verify $d(f^nx,f^ny)\leq\alpha$ for all $n\in\Z$ and $d(x,y)\geq\varepsilon$, contradicting the $[\varepsilon,\alpha]$-expansiveness of $f$.
\end{proof}

\begin{definition}\label{def:unif_alpha_epsilon_exp}
  We say that $f$ is \emph{uniformly $[\varepsilon,\alpha]$-expansive} iff for some $N\in\N$
 \begin{center}
  $x,y\in M$ and $d(f^nx,f^ny)\leq\alpha$ for all $|n|\leq N$ implies $d(x,y)<\varepsilon$,
 \end{center}
 where $0<\varepsilon\leq\alpha$. 
\end{definition}

Clearly, if $f$ is uniformly $[\varepsilon,\alpha]$-expansive then $f$ is $[\varepsilon,\alpha]$-expansive. The converse is also true as we show in the next result which is similar to \cite{Br}*{Theorem 5}, where Bryant proves the property of \emph{uniform expansiveness} for expansive homeomorphisms on compact spaces.

\begin{proposition}\label{prop:unif_alplha_epsilon_exp}
 If $f$ is $[\varepsilon,\alpha]$-expansive then $f$ is uniformly $[\varepsilon,\alpha]$-expansive.
\end{proposition}

\begin{proof}
 If $f$ were not uniformly $[\varepsilon,\alpha]$-expansive then for all $N\in\N$ there exist $x_N,y_N\in M$ such that $d(f^nx_N,f^ny_N)\leq\alpha$ for all $|n|\leq N$ and $d(x_N,y_N)\geq\varepsilon$. Taking limit points $x$ and $y$ of the sequences $(x_N)$ and $(y_N)$, respectively, we have that $d(f^nx,f^ny)\leq\alpha$ for all $n\in\Z$ and $d(x,y)\geq\varepsilon$, which can not be true because $f$ is $[\varepsilon,\alpha]$-expansive. 
\end{proof}

For topological spaces $M$ and $N$ we denote as $\cont(M,N)$ the set of continuous functions from $M$ to $N$ with the \emph{compact-open topology}. If $M$ is compact and $N$ is a metric space this topology is the one induced by the \emph{$C^0$-metric} defined as
$$
d_{C^0}(f,g)=\max\{d(fx,gx):x\in M\}\quad\text{if }f,g\in\cont(M,N).
$$
The subspace of homeomorphisms in $\cont(M,M)$ will be denoted $\homeo(M)$.

\medskip

The following result extends \cite{Le83}*{Lemma 1.1} proved by Lewowicz for the case of an expansive homeomorphism on a compact space.

\begin{proposition}\label{prop:epsilon_alpha_exp_es_abierto}
 The set of $[\varepsilon,\alpha]$-expansive homeomorphisms is open in $\homeo(M)$.
\end{proposition}

\begin{proof}
 Given $0<\varepsilon\leq\alpha$ suppose $f\in\homeo(M)$ is $[\varepsilon,\alpha]$-expansive. By Lemma \ref{lem:gap_abierto} there exists $\delta>0$ such that $f$ is $[\varepsilon,\alpha+\delta]$-expansive, and then by Proposition \ref{prop:unif_alplha_epsilon_exp} $f$ is uniformly $[\varepsilon,\alpha+\delta]$-expansive, that is, for some $N\in\N$ we have
 $$
 \text{if }x,y\in M\text{ and }d(f^nx,f^ny)\leq\alpha+\delta\text{ for all }|n|\leq N\text{ then }d(x,y)<\varepsilon.
 $$
 As the map $\homeo(M)\to\homeo(M)$, $g\mapsto g^n$ is continuous for all $n\in\Z$ we can find a neighborhood $\NN$ of $f$ in $\homeo(M)$ such that $d_{C^0}(f^n,g^n)\leq\nicefrac\delta2$ for all $g\in\mathcal N$ and $|n|\leq N$. Hence, if $g\in\mathcal N$, $x,y\in M$ and $d(g^nx,g^ny)\leq\alpha$ for all $|n|\leq N$, we have
 $$
 d(f^nx,f^ny)\leq d(f^nx,g^nx)+d(g^nx,g^ny)+d(g^ny,f^ny)\leq\nicefrac\delta2+\alpha+\nicefrac\delta2=\alpha+\delta
 $$
 for all $|n|\leq N$, which implies $d(x,y)<\varepsilon$. Then every $g\in\mathcal N$ is $[\varepsilon,\alpha]$-expansive.
\end{proof}

\section{Expansive quotients}\label{sec:expansive_quotients}

In this section we introduce the class of \emph{semi expansive} homeomorphisms to give necessary and sufficient conditions for a quotient of $f$ to be expansive, describing all compact metrizable extensions of an expansive dynamical system in Theorem \ref{teo:cociente_expnsivo}. In \S\ref{subsec:topolog_charact} we obtain a topological version of this theorem, characterizing the compact and non necessarily metrizable extensions of expansive homeomorphisms in Theorem \ref{teo:cocienteexpansivo3}.

\begin{definition}\label{def:semi_expansivo}
 Given a constant $\alpha>0$ we say that $f$ is \emph{$\alpha$-semi expansive} iff $f$ is $[\nicefrac\alpha2,\alpha]$-expansive, that is,
 \begin{center}
  $x,y\in M$ and $d(f^nx,f^ny)\leq\alpha$ for all $n\in\Z$ implies $d(x,y)<\nicefrac\alpha2$.
 \end{center}
 In that case $\alpha$ is called a \emph{semi expansivity constant} for $f$.  
\end{definition}

The next result is a direct consequence of Proposition \ref{prop:epsilon_alpha_exp_es_abierto}.

\begin{proposition}\label{prop:semi_exp_abierto}
 The set of $\alpha$-semi expansive homeomorphisms is open in $\homeo(M)$.
\end{proposition}

For a topological space $M$ and an equivalence relation $R\subseteq M\times M$ we denote $x\sim y$ if $x,y\in M$ are equivalent elements under $R$, $[x]$ the equivalence class of $x\in M$, $M_R=M/R$ the topological quotient space, and $q\colon M\to M_R$, $q(x)=[x]$ if $x\in M$, the canonical map. We say that a subset $A\subseteq M$ is \emph{saturated} iff $A$ contains the class of each of his members. For a homeomorphism $f\colon M\to M$ we say that $R$ is \emph{compatible} with $f$ iff $x,y\in M$ and $x\sim y$ implies $fx\sim fy$. In that case the induced homeomorphism $f_R\colon M_R\to M_R$ given by $f_R[x]=[fx]$ if $x\in M$, will be called the \emph{quotient} of $f$ by $R$. We also say that $f$ is an \emph{extension} of $f_R$.

\begin{definition}\label{def:rel_quiv_de_Lew}
 If $f$ is $\alpha$-semi expansive we define $R(d,\alpha)$ as the equivalence relation on $M$ compatible with $f$ given by
 \begin{center}
  $x\sim y$\quad iff\quad $d(f^nx,f^ny)\leq\alpha$ for all $n\in\Z$,
 \end{center}
 if $x,y\in M$. For a metrizable space $M$, a homeomorphism $f\colon M\to M$ and an equivalence relation $R$ compatible with $f$, we say that the induced quotient homeomorphism $f_R$ is a \emph{Lewowicz quotient} of $f$ iff there exists a compatible metric $d$ on $M$, in the sense that $d$ induces topology of $M$, and $\alpha>0$ such that $f$ is $\alpha$-semi expansive and $R=R(d,\alpha)$.
\end{definition}

Next, we recall the following definition from \cite{Nadler}*{Definition 3.5}.

\begin{definition}\label{def:usd}
 A \emph{decomposition} (also called \emph{partition}) of a topological space $M$, that is a collection $\D$ of pairwise disjoint subsets of $M$ such that $\bigcup\D=M$, is said to be an \emph{upper semi continuous decomposition} iff for every $D\in\D$ and every neighborhood $U$ of $D$ there exists a neighborhood $V$ of $D$ such that if $D'\in\D$ and $D'\cap V\neq\varnothing$ then $D'\subseteq U$.
\end{definition}

The following result is already proved in \cite{CS}*{Remark p.\,323}.

\begin{lemma}\label{lem:cociente_metrizable}
 If $f$ is $\alpha$-semi expansive and $R=R(d,\alpha)$ then $M_R$ is metrizable.
\end{lemma}

\begin{proof}
 By \cite{Nadler}*{Theorem 3.9} it is enough to show that the decomposition of $M$ into equivalence classes is upper semi continuous, that is, for every class $[x]\subseteq M$ and every open set $U\subseteq M$ such that $[x]\subseteq U$ there exists an open set $V\subseteq M$ such that $[x]\subseteq V$ and $[y]\subseteq U$ for all $y\in V$. Suppose on the contrary that there exist a class $[x]\subseteq X$ and an open set $U\subseteq M$, $[x]\subseteq U$, such that for all open sets $V\subseteq M$, $[x]\subseteq V$, there exists $y\in V$ such that $[y]\not\subseteq U$. Then, taking $V=B_{1/n}\bigl([x]\bigr)$ for $n\in\N$, where $B_\varepsilon(A)=\bigcup_{x\in A}B_\varepsilon(x)$ if $\varepsilon>0$ and $A\subseteq M$, we have that for each $n\in\N$ there exists $y_n\in B_{1/n}\bigl([x]\bigr)$ such that $[y_n]\not\subseteq U$.

 Given  $n\in\N$ let $x_n\in[x]$ be such that $d(x_n,y_n)<\nicefrac1n$ and $z_n\in[y_n]$ such that $z_n\notin U$. As $M$ is compact we can assume that $x_n\to\bar x$, $y_n\to y$ and $z_n\to z$. As $x_n\sim x$ for all $n\in\N$ we have $d(f^kx_n,f^kx)\leq\alpha$ for all $k\in\Z$, then taking limit when $n\to+\infty$ we get $d(f^k\bar x,f^kx)\leq\alpha$ for all $k\in\Z$, that is, $\bar x\sim x$. On the other hand, as $y_n\sim z_n$ for all $n\in\N$ we have $d(f^ky_n,f^kz_n)\leq\alpha$ for all $k\in\Z$ and then taking limit when $n\to+\infty$ we conclude that $d(f^ky,f^kz)\leq\alpha$ for all $k\in\Z$, that is $y\sim z$. Finally, as $d(x_n,y_n)<\nicefrac1n$ for all $n\in\N$ we deduce $\bar x=y$. To sum up we proved $x\sim\bar x=y\sim z$, hence $z\in U$ because $[x]\subseteq U$. This contradicts that $z_n\notin U$ for all $n\in\N$ because in that case $z\notin U$ since $U$ is open and $z_n\to z$.
\end{proof}

\begin{remark}\label{obs:diam_clases}
 As a consequence of Lemma \ref{lem:cociente_metrizable} we see that the equivalence classes of the relation $R(d,\alpha)$ of Definition \ref{def:rel_quiv_de_Lew} are compact subsets of $M$. Moreover, if we know that $f$ is $[\varepsilon,\alpha]$-expansive and $\varepsilon\leq\nicefrac\alpha2$ then clearly $f$ is $\alpha$-semi expansive, and the equivalence classes has diameter $\diam([x])<\varepsilon$ if $x\in M$.
\end{remark}

\begin{remark}\label{obs:ugorro_abierto}
 For an equivalence relation $R$ on a topological space $M$, if the canonical map $q\colon M\to M_R$ is closed (maps closed sets to closed sets) then the sets
 $$
 \widehat U=\{x\in M:[x]\subseteq U\}\subseteq M\quad\text and\quad q(\widehat U)\subseteq M_R
 $$
 are open whenever $U\subseteq M$ is open. Indeed, for the openness of $\widehat U$ note that $M\setminus\widehat U=q^{-1}\bigl(q(M\setminus U)\bigr)$, and for $q(\widehat U)$ use that $\widehat U$ is saturated and open. By Lemma \ref{lem:cociente_metrizable} this is true in particular when $R=R(d,\alpha)$ is as in Definition \ref{def:rel_quiv_de_Lew}, because in that case as $M$ is compact and $M_R$ is a Hausdorff space then $q$ is closed.
\end{remark}

\begin{definition}\label{def:cub_exp}
 We say that an open cover $\U$ of $M$ is an \emph{expansivity cover} for $f$ iff $x,y\in M$ and $\{f^nx,f^ny\}\prec\U$ for all $n\in\Z$ implies $x=y$, where the notation $A\prec\U$ means that $A\subseteq U$ for some $U\in\U$.
\end{definition}

For later reference we state the following easy result.

\begin{lemma}\label{lem:cubrimiento_exp}
 $f$ is expansive iff there exists an expansivity cover $\U$ for $f$.
\end{lemma}

In \cite{Le83}*{p. 568} it is suggested that what we called Lewowicz quotient (Definition \ref{def:rel_quiv_de_Lew}) in honor of the author of that article is in fact an expansive homeomorphism. Full details and proofs of this assertion was given later in \cite{CS}, where the authors construct a metric in the quotient space with respect to which the quotient homeomorphism is expansive. In the following result we give an alternative proof of this fact with a shorter but non constructive argument. We also prove the converse of the assertion showing that every expansive quotient is a Lewowicz quotient.

\begin{theorem}\label{teo:cociente_expnsivo}
 Let $M$ be a compact metrizable space, $f\colon M\to M$ a homeomorphism and $f_R\colon M_R\to M_R$ the quotient of $f$ by a compatible equivalence relation $R$ defined on $M$. Then the following conditions are equivalent.
 \begin{enumerate}
  \item $f_R$ is expansive.\footnote{The property of expansiveness for a homeomorphism of a compact metrizable space does not depend on the chosen compatible metric. See for example \cite{AAM}*{Remark 2.10} or apply Lemma \ref{lem:cubrimiento_exp}.} 
  \item $f_R$ is a Lewowicz quotient (Definition \ref{def:rel_quiv_de_Lew}).
 \end{enumerate}
\end{theorem}

\begin{proof}
 \mbox{($1\Rightarrow2$)\;} Pick compatible metrics $d$ and $d_R$ for $M$ and $ M_R$ respectively. Since $f_R$ is expansive there exist an expansivity constant, say $\alpha_R>0$ for $f_R$. Fix any $\alpha>0$ such that $\alpha\leq\alpha_R$, define the constant
 $$
 K=\frac{\alpha}{2\diam(M,d)+1}
 $$
 and consider the new metric $d_1$ on $M$ given by the formula
 $$
 d_1(x,y)=d_R([x],[y])+Kd(x,y),\qquad\text{if }x,y\in M.
 $$
 This is a compatible metric for $M$ because $Kd\leq d_1$ and $q\colon(M,d)\to(M_R,d_R)$ is continuous. We claim that $f$ is $\alpha$-semi expansive with respect to $d_1$ and that the given equivalence relation $R$ coincides with $R(d_1,\alpha)$, this will prove that $f_R$ is a Lewowicz quotient according to Definition \ref{def:rel_quiv_de_Lew}.

 By the choice of $K$ if $x,y\in M$ and $x\sim y$ then $d_1(x,y)=Kd(x,y)<\nicefrac\alpha2$. This implies that $f$ is $\alpha$-semi expansive: if $x,y\in M$ and $d_1(f^nx,f^ny)\leq\alpha$ for all $n\in\Z$, then $d_R(f_R^n[x],f_R^n[y])=d_R([f^nx],[f^ny])\leq d_1(f^nx,f^ny)\leq\alpha\leq\alpha_R$ for all $n\in\Z$, hence, as $\alpha_R$ is an expansivity constant for $f_R$ we get $[x]=[y]$, then $d_1(x,y)<\nicefrac\alpha2$.

 To prove that $R=R(d_1,\alpha)$ denote $x\sim_1y$ to mean equivalence under $R(d_1,\alpha)$. Note that in the preceding paragraph we already showed that $x\sim_1 y$ implies $x\sim y$. Conversely, if $x\sim y$ we have $f^nx\sim f^ny$ for all $n\in\Z$ because $R$ is compatible with $f$. Then $d_1(f^nx,f^ny)=Kd(f^nx,f^ny)<\nicefrac\alpha2\leq\alpha$ for all $n\in\Z$, that is, $x\sim_1 y$.

 \mbox{($2\Rightarrow\!1$)\;} Assume now that $f_R$ is a Lewowicz quotient. Then by Definition \ref{def:rel_quiv_de_Lew} we know that there exist a metric $d$ on $M$ and $\alpha>0$ such that $f$ is $\alpha$-semi expansive and $R=R(d,\alpha)$. As by Lemma \ref{lem:cociente_metrizable} we have that $M_R$ is metrizable, in order to prove that $f_R$ is expansive it is enough to show an expansivity cover (Lemma \ref{lem:cubrimiento_exp}).

 For each $x\in M$ let $U_x=B_{\alpha/2}(x)$ and $\widehat U_x=\{y\in M:[y]\subseteq U_x\}$. By Remark \ref{obs:ugorro_abierto} we know that $\widehat U_x$ and $q(\widehat U_x)$ are open sets for all $x\in M$. Moreover, for each $x\in M$ we have $x\in\widehat U_x$ because by Remark \ref{obs:diam_clases} $\diam([x] )<\nicefrac\alpha2$ and hence $[x]\subseteq U_x$. Consequently $\U_R=\{q(\widehat U_x):x\in M\}$ is an open cover of $M_R$.

 We claim that $\U_R$ is an expansivity cover for $f_R$. Indeed, if  $x,y\in M$ and $\{f_R^n[x],f_R^n[y]\}\prec\U_R$ for all $n\in\Z$ then for each $n\in\Z$ we have $\{[f^nx],[f^ny]\}\subseteq q\bigl(\widehat U(z_n)\bigr)$ for suitable $z_n\in M$. As $\widehat U(z_n)$ is saturated we get $\{f^nx,f^ny\}\subseteq\widehat U(z_n)\subseteq B_{\alpha/2}(z_n)$, and then $d(f^nx,f^ny)<\alpha$ for all $n\in\Z$. Hence $[x]=[y]$ as desired.
\end{proof}

\begin{remark}\label{obs:expgordo}
 In the proof of \mbox{($1\Rightarrow2$)} in the previous proposition the semi expansivity constant $\alpha>0$ of $f$ can be chosen arbitrarily small compared with the expansivity constant $\alpha_R$ of $f_R$. In addition, taking a smaller constant $K$ we see that we can construct a metric on $M$ such that $f$ is $[\varepsilon,\alpha]$-expansive with $\varepsilon>0$ arbitrarily small relative to the chosen $\alpha>0$. Finally, note that the metrics $d_1$ constructed in this way verifies $d_R([x],[y])\leq d_1(x,y)$ for all $x,y\in M$. 
\end{remark}

As commented in the Introduction the first type of examples of semi expansive homeomorphisms are the perturbations of expansive homeomorphisms. Later, in \S4.2, the Lemma \ref{lem:expgordo_shift} will show another big source of examples of different kind, and to which we will apply our results. Next, to give some insight, we introduce two toy examples of semi expansive homeomorphisms.

\begin{example}\label{ex:semiexpansive1}
 Let $f\colon M\to M$ be an expansive homeomorphism with expansivity constant $\alpha$, consider the product space $M_1=M\times[0,1]$ and the homeomorphism $f_1\colon M_1\to M_1$ given by $f_1(x,s)=(fx,s)$ if $x\in M$, $s\in[0,1]$. Let $d_1$ be the metric on $M_1$ defined as\quad $d_1\bigl((x,s),(y,t)\bigr)=d(x,y)+\frac\alpha3|t-s|$,\quad if $x,y\in M$, $s,t\in[0,1]$.
 Then it is easily checked that $f_1$ is $\alpha$-semi expansive. In this case the equivalence classes of the associated equivalence relation are of the form $\{x\}\times[0,1]$, $x\in M$, the corresponding expansive quotient system is (conjugated to) $f$, and the canonical quotient map is (identified with) the projection on the first factor of $M\times[0,1]$.
\end{example}

\begin{example}\label{ex:semiexpansive2}
 For $i=1,2$ let $f_i\colon(M_i,d_i)\to(M_i,d_i)$ be an expansive homeomorphism with expansivity constant $1$. For simplicity assume also that $\diam M_1=\diam M_2=2$. Consider the product space $M=M_1\times M_2$ and the homeomorphism $f\colon M\to M$, $f(x_1,x_2)=(f_1x_1,f_2x_2)$ if $x_1\in M_1$, $x_2\in M_2$. Let $d'$ be the metric on $M$ given by \quad $d'\bigl((x_1,x_2),(y_1,y_2)\bigr)=2d_1(x_1,y_1)+\frac13d_2(x_2,y_2)$,\quad if $x_1,y_1\in M_1$, $x_2,y_2\in M_2$. Then, in a similar fashion as in Example \ref{ex:semiexpansive1} we have that the system $(M,d',f)$ is $2$-semi expansive and the associated expansive quotient system is (conjugated to) $f_1$. Analogously, one can introduce a metric $d''$ on $M$ such that $(M,d'',f)$ y also $2$-expansive but this time the associated expansive quotient is (conjugated to) $f_2$. Therefore, we see that even though the systems $(M,d',f)$ and $(M,d'',f)$ are conjugated and both $2$-semi expansive, they are semi expansive in a ``different way'', because the equivalence classes are of different kind in each case.
\end{example}

\subsection{\sc Topological characterization}\label{subsec:topolog_charact}

In Theorem \ref{teo:cociente_expnsivo} we given a characterization of those compact metrizable dynamical systems which are extension of an expansive system. In this subsection we consider this problem from a purely topological point of view, that is, we drop the assumption of metrizability of the extension. In Theorem \ref{teo:cocienteexpansivo3} we solve this problem giving necessary and sufficient conditions for an arbitrary compact system to be an extension of an expansive system.

\medskip

Only for this subsection we change our context convention and assume that $X$ is a compact topological space and $f\colon X\to X$ a homeomorphism.

\begin{definition}\label{def:rel_RC}
 For a cover $\C$ of $X$ let $R(\C)\subseteq X\times X$ be the relation on $X$, also denoted $\sim_\C$, given by the formula
 $$
 x\sim_\C y\quad\text{iff}\quad\{f^nx,f^ny\}\prec\C\text{ for all }n\in\Z,
 $$
 if $x,y\in X$, where as before the notation $A\prec\C$ means that $A\subseteq C$ for some $C\in\C$.  Given $x\in X$ we also define the sets
 $$
 [x]_\C=\{y\in X:x\sim_\C y\}\quad\text{and}\quad X_\C=\{[x]_\C:x\in X\}.
 $$
\end{definition}

Note that the relation $R(\C)$ of Definition \ref{def:rel_RC} is a symmetric and reflexive relation on $X$, but fails in general to be transitive. The relation $R(\C)$ will be transitive, and hence an equivalence relation, precisely when the collection $X_\C$ is a partition of $X$, in which case $X/{\sim_\C}=X_\C$.

\medskip

In the following Remark \ref{obs:conj_dirigido} and in Lemma \ref{lem:semicontinuidad} we will use the notions of \emph{directed set}, \emph{nets}, \emph{subnets} and \emph{convergence} as presented in  \cite{Kelley}*{Chapter 2}.

\begin{remark}\label{obs:conj_dirigido}
 If $(\Lambda,\leq)$ is a directed set and $\Lambda=\Lambda_1\cup\cdots\cup\Lambda_r$ then at least one of the sets $\Lambda_i$ verifies the following property: for every $\lambda\in\Lambda$ there exists $\mu\in\Lambda_i$ such that if $\nu\in\Lambda_i$ and $\mu\leq\nu$ then $\lambda\leq\nu$.

 Indeed, if the assertion is false then for all $i\in\{1,\ldots,r\}$ there exists $\lambda_i\in\Lambda$ such that for all $\mu\in\Lambda_i$ there exists $\nu\in\Lambda_i$ such that $\mu\leq\nu$ but $\lambda_i\not\leq\nu$. As $\Lambda$ is a directed set, there exists $\mu\in\Lambda$ such that $\lambda_i\leq\mu$ for all $i\in\{1,\ldots,r\}$. Let  $i_0\in\{1,\ldots,r\}$ be such that $\mu\in\Lambda_{i_0}$. Then, according to our assumption there exists $\nu\in\Lambda_{i_0}$ such that $\mu\leq\nu$ and $\lambda_{i_0}\not\leq\nu$, which is a contradiction because $\lambda_{i_0}\leq\mu$.

 As a consequence we see that if $(x_\lambda)_{\lambda\in\Lambda}$ is a net in $X$ then $(x_\lambda)_{\lambda\in\Lambda_i}$ is a subnet for some $i\in\{1,\ldots,r\}$.
\end{remark}

Let $2^X$ be the collection of all non-empty subsets of $X$. A map $F\colon X\to2^X$ is said to be \emph{upper semi continuous} iff for all $x\in X$ and all neighborhoods $U$ of $F(x)$ there exists a neighborhood $V$ of $x$ such that $F(y)\subseteq U$ if $y\in V$ (see \cite{Aub}*{Definition 1.4.1}). It is easy to check that a decomposition $\D$ of $X$ is upper semi continuous (see Definition \ref{def:usd}) iff the map $X\to2^X$ that associates to each $x\in X$ the unique element $D\in\D$ such that $x\in D$ is upper semi continuous.

\begin{lemma}\label{lem:semicontinuidad}
 If $\C$ is a finite closed cover of $X$ then the map $X\to2^X$, $x\mapsto[x]_\C$, is upper semi continuous. If in addition $R(\C)$ is an equivalence relation then $X_\C$ is an upper semi continuous decomposition.  
\end{lemma}

\begin{proof}
 To prove the first assertion suppose on the contrary that the map $X\to2^X$, $x\mapsto[x]_\C$, is not upper semi continuous. Then there exist $x\in X$ and a neighborhood $U$ of $[x]_\C$ such that for every neighborhood $V$ of $x$ there exists $x'\in V$ such that $[x']_\C\not\subseteq U$. Picking a point $x'$ as before in each neighborhood $V$ of $x$ one can construct a net $(x_\lambda)_{\lambda\in\Lambda}$ converging to $x$ such that $[x_\lambda]_\C\not\subseteq U$ for all $\lambda\in\Lambda$. Then there exists a net $(y_\lambda)_{\lambda\in\Lambda}$ such that $x_\lambda\sim_\C y_\lambda$ and $y_\lambda\notin U$ for all  $\lambda\in\Lambda$. As $X$ is compact and $U$ is open taking a subnet if necessary we can assume that $y_\lambda\to y\notin U$.

 For each fixed $n\in\Z$ we have that for all $\lambda\in\Lambda$ there exists $C\in\C$ such that $\{f^nx_\lambda,f^ny_\lambda\}\subseteq C$ because $x_\lambda\sim_\C y_\lambda$. Hence $\Lambda$ can be written as a finite union $\Lambda=\bigcup_{C\in\C}\Lambda_C$ where $\Lambda_C=\bigl\{\lambda\in\Lambda:\{f^nx_\lambda,f^ny_\lambda\}\subseteq C\bigr\}$. Then, by Remark \ref{obs:conj_dirigido} there exists $C\in\C$ such that $(x_\lambda)_{\lambda\in\Lambda_C}$ and $(y_\lambda)_{\lambda\in\Lambda_C}$ are subnets of $(x_\lambda)_{\lambda\in\Lambda}$ and $(y_\lambda)_{\lambda\in\Lambda}$, respectively. For this $C$ we have $\{f^nx_\lambda,f^ny_\lambda\}\subseteq C$ for all $\lambda\in\Lambda_C$, then as $C$ is closed taking limits we obtain $\{f^nx,f^ny\}\subseteq C$, that is $\{f^nx,f^ny\}\prec\C$. As this is true for all $n\in\Z$ we conclude that $x\sim_\C y$, and therefore $y\in[x]_\C\subseteq U$, which is a contradiction because we previously had $y\notin U$.

 The last assertion follows from the comment made before this Lemma.
\end{proof}

In the next Definition \ref{def:Usemiexp} and Proposition \ref{teo:cocienteexpansivo3} for a collection $\C$ of subsets of $X$ we use the notation $\xoverline\C=\{\xoverline C:C\in\C\}$, where $\xoverline C$ stands for the closure of $C$, and 
\begin{center}
 $\C^k=\{C_1\cup\cdots\cup C_k:C_1,\ldots,C_k\in\C,\,C_i\cap C_{i+1}\neq\varnothing\text{ for }1\leq i<k\}$,
\end{center}
if $k\in\N$. Note that if $\C$ is a closed cover then $\C^k$ is a closed cover too.

\begin{definition}\label{def:Usemiexp}
 For a finite open cover $\U$ of $X$ we say that $f$ is \emph{$\U$-semi expansive} iff $x,y\in X$ and $x\sim_{\text{$\xoverline\U$}{}^{^4}}y$ impies $x\sim_\U y$, that is, $R(\xoverline\U^4)\subseteq R(\U)$. 
\end{definition}

The above definition will play the role of Definition \ref{def:semi_expansivo} in the purely topological setting, although it resembles the definition of $[\nicefrac\alpha4,\alpha]$-expansiveness in the metrizable case instead of $[\nicefrac\alpha2,\alpha]$-expansiveness.

\begin{remark}\label{obs:R(U)equiv}
 If $f$ is $\U$-semi expansive, as we always have $R(\U)\subseteq R(\xoverline\U^4)$ because $\U\prec\xoverline\U^4$, we see that $R(\U)=R(\xoverline\U^4)$. Moreover, this is an equivalence relation as the following verification of the transitive law shows: $R(\U)\circ R(\U)\subseteq R(\U^2)\subseteq R(\xoverline\U^2)\subseteq R(\xoverline\U^4)=R(\U)$, where $\circ$ denotes the usual composition of relations. Note also that $R(\U)$ is compatible with $f$.
\end{remark}

\begin{remark}\label{obs:generador}
 For later use let us recall some facts from \cite{KR}. A \emph{generator} for $f$ is a finite open cover $\U$ of $X$ such that $\bigcap_{n\in\Z}f^n\xoverline U_n$ contains at most one point for every bi-sequence $(U_n)_{n\in\Z}$ of members of $\U$ (see \cite{KR}*{Definition 2.4}). Note that this is equivalent to: if $x,y\in X$ and $\{f^nx,f^ny\}\prec\xoverline\U$ for all $n\in\Z$ then $x=y$. In \cite{KR}*{Theorem 3.2} it is proved that $f$ has a generator iff $f$ is expansive.
\end{remark}

Next we introduce the main result of this subsection, a topological characterization of general extensions of expansive systems in the compact case.

\begin{theorem}\label{teo:cocienteexpansivo3}
 Let $f\colon X\to X$ be a homeomorphism of a compact topological space $X$, $R$ an equivalence relation on $X$ compatible with $f$, $X_R=X/R$ the quotient space and $f_R\colon X_R\to X_R$ the induced homeomorphism. Then the following conditions are equivalent.
 \begin{enumerate}
  \item $f_R$ is expansive.
  \item $f$ is $\U$-semi expansive and $R=R(\U)$ for some finite open cover $\U$ of $X$.
\end{enumerate}
\end{theorem}

\begin{proof}
 \mbox{$1)\Rightarrow2).\,$} Let $d$ be a compatible metric for $X_R$ and $\alpha>0$ an expansivity constant for $f_R$ relative to $d$. Consider a finite open cover $\U_R$ of $X$ whose members have diameter less than $\nicefrac\alpha4$ and define $\U=q^{-1}(\U_R)$, where $q\colon X\to X_R$, $q(x)=[x]$ if $x\in X$, is the canonical projection.

 We claim that $f$ is $\U$-semi expansive. Indeed, if $x,y\in X$ and $x\sim_{\text{$\xoverline\U$}{}^{^4}}y$, then for each $n\in\Z$ there exist $U_1,U_2,U_3,U_4\in\U$ such that $f^nx\in\xoverline U_{\hspace{-.4mm}1}$, $\xoverline U_{\hspace{-.4mm}i}\cap\xoverline U_{\hspace{-.4mm}i+1}\neq\varnothing$ for $1\leq i\leq3$ and $f^ny\in\xoverline U_{\hspace{-.4mm}4}$. Hence $f_R^n[x]=[f^nx]\in q(\xoverline U_{\hspace{-.4mm}1})$, $q(\xoverline U_{\hspace{-.4mm}i})\cap q(\xoverline U_{\hspace{-.4mm}i+1})\neq\varnothing$ for $1\leq i\leq3$ and $f_R^n[y]=[f^ny]\in q(\xoverline U_{\hspace{-.4mm}4})$. As $q(\xoverline U_{\hspace{-.4mm}i})\subseteq \xoverline{q(U_i)}$ and $q(U_i)\in\U_R$ has diameter less than $\nicefrac\alpha4$ for $1\leq i\leq4$ we deduce that $d(f_R^n[x],f_R^n[y])\leq\alpha$ for all $n\in\Z$. Therefore, as $f_R$ is $\alpha$-expansive we obtain $[x]=[y]$, and then $x\sim_\U y$ because the members of $\U$ are $R$-saturated.

 Finally note that $R=R(\U)$ because if $x,y\in X$ and $x\sim_\U y$ then   clearly $x\sim_{\text{$\xoverline\U$}{}^{^4}}y$, hence as we already showed in the previous paragraph $[x]=[y]$, that is, $x\sim y$, where ``$\sim$'' means equivalence under $R$. Conversely if $x\sim y$ then $x\sim_\U y$ because the members of $\U$ are $R$-saturated.

 \mbox{$2)\Rightarrow1).\,$} As $f$ is $\U$-semi expansive we know that  $R(\xoverline\U^4)=R(\U)=R$ (see Remark \ref{obs:R(U)equiv}). Then, as $\xoverline\U^4$ is a finite closed cover of $X$, by Lemma \ref{lem:semicontinuidad} $X_R$ is an upper semi continuous decomposition of $X$. Applying \cite{Nadler}*{Proposition 3.7} we obtain that the canonical map $q\colon X\to X_R$ is a closed map.

 To prove that $f_R$ is expansive we will show a generator $\U_R$. To this end for each $x\in X$ consider $U(x)=\bigcup\{U\in\U:x\in U\}$ and $\widehat U(x)=\{y\in X:[y]\subseteq U(x)\}$, and define the finite collection $\U_R=\bigl\{q\bigl(\widehat U(x)\bigr):x\in X\bigr\}$. Note that $x\in\widehat U(x)$ for all $x\in X$, because if $y\in[x]=[x]_\U$ then $\{x,y\}\prec\U$ and hence $y\in U(x)$, that is, $[x]\subseteq U(x)$. Hence we deduce that $\U_R$ is a cover of $X_R$. Moreover, by Remark \ref{obs:ugorro_abierto} we conclude that $\U_R$ is an open cover.

 We claim that $\U_R$ is a generator for $f_R$. By Remark \ref{obs:generador} it is enough to show that if $x,y\in X$ and $[x]\sim_{\text{$\xoverline{\U_R}$}}[y]$ then $[x]=[y]$. Suppose that $x,y\in X$ are points such that $\{f_R^n[x],f_R^n[x]\}\prec\xoverline{\U_R}$ for all $n\in\Z$. For each $n\in\Z$ let $z_n\in X$ such that $\{[f^nx],[f^ny]\}\subseteq q\bigl(\widehat U(z_n)\bigr)^-$. Note that $q\bigl(\widehat U(z_n)\bigr)^-\subseteq q\bigl(U(z_n)\bigr)^-=q\bigl(U(z_n)^-\bigr)$, where in the last equality we used that $q$ is a closed map as we already showed before. Then we can take $x_n,y_n\in U(z_n)^-$ such that $f^nx\sim x_n$ and $f^ny\sim y_n$ for all $n\in\Z$. As $\U$ is finite, for all $z\in X$ we have $U(z)^-=\bigl(\bigcup\{U\in\U:x\in U\}\bigr)^-=\bigcup\{\xoverline U:x\in U\in\U\}$, then $\{x_n,z_n\}\prec\xoverline\U$ and $\{y_n,z_n\}\prec\xoverline\U$ because $x_n,y_n\in U(z_n)^-$. In addition we have  $\{f^nx,x_n\}\prec\U$ and $\{f^ny,y_n\}\prec\U$ because $f^nx\sim x_n$, $f^ny\sim y_n$ and $R=R(\U)$. Then we conclude that $\{f^nx,f^ny\}\prec\xoverline\U^4$ for all $n\in\Z$. Therefore, $x\sim_{\text{$\xoverline\U$}{}^{^4}}y$, and hence $[x]=[y]$ because $R=R(\xoverline\U^4)$.
\end{proof}

\section{Anosov quotients}\label{sec:anosov_quotients}

In this section we give a characterization of the metrizable extensions of Anosov systems (see Definition \ref{def:shadowing}) in the compact case in Theorem \ref{teo:cociente_anosov}. To do that we introduce in Definition \ref{def:semi_anosov} the class of \emph{semi Anosov homeomorphisms}. Later, in \S\ref{subsec:estabilidad_gorda} and \S\ref{subsec:envovente} we show two applications of this theorem (see Theorems \ref{teo:estabilidad} and \ref{teo:envolvente}). As usual $(M,d)$ will denote a compact metric space and $f\colon M\to M$ a homeomorphism.

\begin{definition}\label{def:shadowing}
 Let $\xi=(x_n)_{n\in\Z}$ a bi--sequence in $M$. Given $\delta>0$, $\xi$ is called \emph{$\delta$-pseudo orbit} iff $d(fx_n,x_{n+1})<\delta$ for all $n\in\Z$. For $\varepsilon>0$ we say that $x\in M$ \emph{$\varepsilon$-shadows} $\xi$ iff $d(f^nx,x_n)<\varepsilon$ for all $n\in\Z$.

 Let $\Set\subseteq M^\Z$ be a set of bi--sequences in $M$. Given $\delta>0$ and $\varepsilon>0$ we say that $f$ has the \emph{$\varepsilon-\delta$ shadowing property on $\Set$} iff every $\delta$-pseudo orbit belonging to $\Set$ can be $\varepsilon$-shadowed. We say that $f$ has the \emph{shadowing property on $\Set$} iff for every $\varepsilon>0$ there exists $\delta>0$ such that $f$ has the $\varepsilon-\delta$ shadowing property on $\Set$. In the previous cases if $\Set=M^\Z$ we use the same terminology omitting the reference to $\Set$. If $f$ is expansive and has the shadowing property then $f$ is called \emph{Anosov homeomorphism}.
\end{definition}

We extend some of the above definitions to the topological setting as follows.

\begin{definition} 
 Let $f\colon M\to M$ be a homeomorphism of a topological space $M$ and $\xi=(x_n)_{n\in\Z}$ a bi--infinite sequence in $M$. Given an open cover $\V$ of $M$ we say that $\xi$ is a \emph{$\V$-pseudo orbit} iff $\{fx_n,x_{n+1}\}\prec\V$ for every $n\in\Z$, where for $A\subseteq M$ the notation  $A\prec\V$ means that $A\subseteq V$ for some $V\in\V$. Given an open cover $\U$ of $M$ we say that $x\in M$ \emph{$\U$-shadows} $\xi$ iff $\{f^nx,x_n\}\prec\U$ for all $n\in\Z$.
\end{definition}

In the next auxiliary result we will use the notion of generator. See Remark \ref{obs:generador}.

\begin{lemma}\label{lem:shadowing_top} 
 For $\Set\subseteq M^\Z$ the following conditions are equivalent.
 \begin{enumerate}
  \item $f$ is expansive and has the shadowing property on $\Set$.
  \item There exist a generator $\U$ for $f$ and an open cover $\V$ of $M$ such that every $\V$-pseudo orbit belonging to $\Set$ can be $\U$-shadowed.
 \end{enumerate}
\end{lemma}

\begin{proof}
 \mbox{($1\Rightarrow2$)\;} Let $\alpha>0$ be an expansivity constant for $f$ and $\delta>0$ such that $f$ has the $\nicefrac\alpha2-\delta$ shadowing property on $\Set$. Take as $\U$ any finite cover of $M$ by balls of radius $\nicefrac\alpha2$, and as $\V$ the cover by all balls of radius $\nicefrac\delta2$. It is straightforward to check that these covers fulfils the desired requirements.

 \mbox{($2\Rightarrow\!1$)\;} As $\U$ is a generator by \cite{KR}*{Theorem 3.2} we know that $f$ is expansive. To prove the last part of the statement suppose on the contrary that $f$ does not have the shadowing property on $\Set$. Then there exists an $\varepsilon>0$ such that for every $k\in\N$ there is a $\nicefrac1k$-pseudo orbit $\xi_k=(x_n^k)_{n\in\Z}$ belonging to $\Set$ not $\varepsilon$-shadowable. Let $k_0\in\N$ such that $\nicefrac1{k_0}$ is a Lebesgue number for $\V$. Then, for each $k\geq k_0$ we have that $\xi_k$ is a $\V$-pseudo orbit, therefore there exists $z_k\in M$ that $\U$-shadows it.

 As $\xi_k$ is not $\varepsilon$-shadowable for all $k\geq k_0$ we have  $d(f^{n_k}z_k,x_{n_k}^k)\geq\varepsilon$ for suitable $n_k\in\Z$. Rearranging the indices we can assume that $n_k=0$ for each $k\geq k_0$. Taking subsequences we can also assume that $x_0^k\to x$ and $z_k\to z$. As $d(z_k,x_0^k)\geq\varepsilon$ for all $k\geq k_0$ we see that $d(z,x)\geq\varepsilon$ and then $z\neq x$. On the other hand it is easy to show that $x^k_n\to f^nx$ for all $n\in\Z$, that is, the pseudo orbits $\xi_k$ converge pointwise to the orbit of $x$. Then, as $\{f^nz_k,x_n^k\}\prec\U$ for all $n\in\Z$ and $k\geq k_0$ we deduce that $\{f^nz,f^nx\}\prec\overline\U$ for all $n\in\Z$. This contradicts that $\U$ is a generator.
\end{proof}

\begin{proposition}\label{prop:sombreado_cociente}
 Let $\Set\subseteq M^\Z$ and suppose that $f$ is $[\nicefrac\delta4,\alpha]$-expansive and has the $\nicefrac\alpha4-\delta$ shadowing property on $\Set$ for some $0<\delta\leq\nicefrac\alpha4$. Let $R=R(d,\alpha)$ as in Definition \ref{def:rel_quiv_de_Lew}, $M_R=M/R$ and $f_R$ the induced (expansive) homeomorphism of $M_R$. Then $f_R$ has the shadowing property on $\Set_R=q(\Set)=\bigl\{\bigl(q(x_n)\bigr)_{n\in\Z}:(x_n)_{n\in\Z}\in\Set\bigr\}$, where $q\colon M\to M_R$ denotes the canonical map.
\end{proposition}

\begin{proof}
 Firstly note that as $\nicefrac\delta4\leq\nicefrac\alpha2$ and $f$ is $[\nicefrac\delta4,\alpha]$-expansive by Remark \ref{obs:diam_clases} we have that $f$ is $\alpha$-semi expansive and the diameter of the $R$-equivalence classes verify $\diam([x])<\nicefrac\delta4\leq\nicefrac\alpha{16}$ if $x\in M$. By Lemma \ref{lem:shadowing_top} to prove the thesis it is enough to show a generator $\U_R$ for $f_R$ and an open cover $\V_R$ of $M_R$ such that every $\V_R$-pseudo orbit belonging to $\Set_R$ is $\U_R$-shadowable.

 Let $F\subseteq M$ be a finite set such that $\bigcup_{x\in F}B_{\alpha/8}(x)=M$, for $x\in F$ define $U_x=B_{7\alpha/16}(x)$ and $\widehat U_x=\{y\in M:[y]\subseteq U_x\}$, and consider the families of sets $\U=\{\widehat U_x:x\in F\}$ and $\U_R=\{q(\widehat U_x):x\in F\}$. By Remark \ref{obs:ugorro_abierto} these are families of open sets. Moreover, for all $x\in M$ we can pick $y\in F$ such that $d(x,y)<\nicefrac\alpha8$, hence as $\diam([x])<\nicefrac\alpha{16}$ we have $[x]\subseteq U_y$ ($\nicefrac\alpha8+\nicefrac\alpha{16}\leq\nicefrac{7\alpha}{16}$), and therefore $x\in\widehat U_y\in\U$. This proves that $\U$ and then $\U_R$ are covers.

 We claim that $\U_R$ is a generator for $f_R$. To see this, note first that if $z\in F$ then 
 \begin{equation}\label{ec:001}
  q^{-1}\bigl(q(\widehat U_z)^-\bigr)\subseteq B_{\alpha/2}(z).
 \end{equation}
 Indeed, as $q$ is a closed map (because $M_R$ is metrizable by Lemma \ref{lem:cociente_metrizable} and $M$ is compact) it holds that $q(\widehat U_z)^-=q(\widehat U_z^-)$. Then, for all $w\in q^{-1}\bigl(q(\widehat U_z)^-\bigr)$ we have $[w]\in q(\widehat U_z)^-=q(\widehat U_z^-)$, and as $\widehat U_z\subseteq U_z=B_{7\alpha/16}(z)$ we have $[w]=[w']$ for some $w'\in M$ such that $d(z,w')\leq\nicefrac{7\alpha}{16}$. Hence $w\in B_{\alpha/2}(z)$ because $\diam([w])<\nicefrac\alpha{16}$ and $\nicefrac{7\alpha}{16}+\nicefrac{\alpha}{16}=\nicefrac\alpha2$.

 Now suppose that $x,y\in M$ and $\{f_R^n[x],f_R^n[y]\}\prec\overline\U_{\!\!R}$ for all $n\in\Z$. Hence, for each $n\in\Z$ we have $\{[f^nx],[f^ny]\}\subseteq q(\widehat U_{z_n})^-$ for suitable $z_n\in F$. Applying the relation (\ref{ec:001}) we get $\{f^nx,f^ny\}\subseteq B_{\alpha/2}(z_n)$, and therefore $d(f^nx,f^ny)<\alpha$ for all $n\in\Z$. This implies that $[x]=[y]$, proving that $\U_R$ is a generator.

 For each $x\in M$ let $V_x=B_{\delta/2}(x)$ and $\widehat V_x=\{y\in M:[y]\subseteq V_x\}$, and define the families $\widehat\V=\{\widehat V_x:x\in M\}$ and $\V_R=\{q(\widehat V_x):x\in M\}$. One can check in a similar fashion as done before with $\U$ and $\U_R$ that $\V$ and $\V_R$ are open covers. We affirm that every $\V_R$-pseudo orbit belonging to $\Set_R$ can be $\U$-shadowed.

 In order to prove the claim suppose given a $\V_R$-pseudo orbit $\xi_R\in\Set_R$. This pseudo orbit is of the form $\xi_R=([x_n])_{n\in\Z}$ for some $\xi=(x_n)_{n\in\Z}\in\Set$, and for all $n\in\Z$ we have $\{[fx_n],[x_{n+1}]\}=\{f_R[x_n],[x_{n+1}]\}\subseteq q(\widehat V_{y_n})$ for suitable $y_n\in M$. As $\widehat V_{y_n}$ is saturated it follows that $\{fx_n,x_{n+1}\}\subseteq\widehat V_{y_n}\subseteq V_{y_n}=B_{\delta/2}(y_n)$ for all $n\in\Z$. Then we see that $\xi$ is a $\delta$-pseudo orbit. Hence, by hypothesis, there exists $x\in M$ that $\nicefrac\alpha4$-shadows $\xi$, that is $d(f^nx,x_n)<\nicefrac\alpha4$ for all $n\in\Z$.

 For each $n\in\Z$ take $z_n\in F$ such that $d(x_n,z_n)<\nicefrac\alpha8$. As $\diam([f^nx])<\nicefrac\alpha{16}$ for all $n\in\Z$ and $\nicefrac\alpha4+\nicefrac\alpha8+\nicefrac\alpha{16}=\nicefrac{7\alpha}{16}$, we deduce $f_R^n[x]=[f^nx]\subseteq U_{z_n}=B_{7\alpha/16}(z_n)$, hence $f_R^n[x]\in q(\widehat U_{z_n})$ for all $n\in\Z$. Similarly one can check that $[x_n]\in q(\widehat U_{z_n})$ for all $n\in\Z$. Then we proved that $\xi_R$ is $\U_R$-shadowed by $[x]$, and we are done.
\end{proof}

\begin{definition}\label{def:semi_anosov}
 Given $\alpha>0$ we say that $f$ is an \emph{$\alpha$-semi Anosov} homeomorphism iff $f$ is $[\nicefrac\delta4,\alpha]$-expansive and has the $\nicefrac\alpha4-\delta$ shadowing property for some $0<\delta\leq\nicefrac\alpha4$.
\end{definition}

\begin{theorem}\label{teo:cociente_anosov}
 Let $M$ be a compact metrizable space, $f\colon M\to M$ a homeomorphism, $R$ an equivalence relation on $M$ compatible with $f$, $M_R=M/R$ and $f_R$ the induced quotient homeomorphism of $M_R$. The following statements are equivalent.
 \begin{enumerate}
  \item $f_R$ is an Anosov homeomorphism.
  \item $f$ is $\alpha$-semi Anosov and $R=R(d,\alpha)$ for some compatible metric $d$ on $M$ and $\alpha>0$.
 \end{enumerate}
\end{theorem}

\begin{proof}
 \mbox{($1\Rightarrow2$)\;} As in the proof of Theorem \ref{teo:cociente_expnsivo} pick compatible metrics $d$ and $d_R$ on $M$ and $M_R$, respectively, let $\alpha_R>0$ be an expansivity constant for $f_R$, choose $\alpha>0$ such that $\alpha\leq\alpha_R$, and define a compatible metric $d_1$ on $M$ by the formula
 $$
 d_1(x,y)=d_R([x],[y])+Kd(x,y),\qquad\text{if }x,y\in M,
 $$
 but this time taking the constant $K$ given by 
 $$
 K=\frac{\delta}{4\diam(M,d)+1},
 $$
 where $0<\delta\leq\nicefrac\alpha4$ is such that $f_R$ has the $\nicefrac\alpha8-\delta$ shadowing property. In a similar fashion as in the cited proof one can check that $f$ is $[\nicefrac\delta4,\alpha]$-expansive relative to $d_1$.

 We claim that $f$ also have the $\nicefrac\alpha4-\delta$ shadowing property with respect to $d_1$. To prove that suppose $\xi=(x_n)_{n\in\Z}$ is a given $\delta$-pseudo orbit of $f$ relative to $d_1$, that is, $d_1(fx_n,x_{n+1})<\delta$ for all $n\in\Z$. Then, from the formula defining $d_1$ we have that $d_R(f_R[x_n],[x_{n+1}])=d_R([fx_n],[x_{n+1}])\leq d_1(fx_n,x_{n+1})<\delta$ for all $n\in\Z$. Hence $\xi_R=([x_n])_{n\in\Z}$ is a $\delta$-pseudo orbit of $f_R$. Consequently, by the choice of $\delta$ there exists $x\in M$ such that $\xi_R$ is $\nicefrac\alpha8$-shadowed by $[x]$, that is, $d_R(f_R^n[x],[x_n])<\nicefrac\alpha8$ for all $n\in\Z$. Now, for each $n\in\Z$ we estimate
 \begin{equation*}
 \begin{split}
 d_1(f^nx,x_n)&=d_R([f^nx],[x_n])+Kd(f^nx,x_n)\\
  &= d_R(f_R^n[x],[x_n])+Kd(f^nx,x_n)\\
  &<\nicefrac\alpha8+K\diam(M,d)\\
  &<\nicefrac\alpha8+\nicefrac\delta4\\
  &<\nicefrac\alpha4.
 \end{split}
 \end{equation*}
 This shows that $\xi$ is $\nicefrac\alpha4$-shadowed by $x$ and the proof is complete.

 \mbox{($2\Rightarrow\!1$)\;} It follows from Proposition \ref{prop:sombreado_cociente} with $\Set=M^\Z$. 
\end{proof}

\begin{remark}\label{obs:anosovgordo}
 In the proof of \mbox{($1\Rightarrow2$)} in the previous proposition $\alpha>0$ can be chosen arbitrarily small compared with the expansivity constant $\alpha_R$ of $f_R$. Moreover, we can take $\delta>0$ arbitrarily small relative to the chosen $\alpha>0$. Note also that the constructed metric $d_1$ verifies $d_R([x],[y])\leq d_1(x,y)$ for all $x,y\in M$. 
\end{remark}

\begin{remark}
 In view of Lemma \ref{lem:shadowing_top} where the expansiveness and shadowing property of a system is expressed in a topological fashion, using the open covers of the space instead of any metric, it seems very possible to adapt the techniques of \S\ref{subsec:topolog_charact} to obtain a topological version of Theorem \ref{teo:cociente_anosov}, characterizing arbitrary compact and non necessarily metrizable extensions of Anosov systems.  
\end{remark}

\subsection{A stability result}\label{subsec:estabilidad_gorda}

As an application of Theorem \ref{teo:cociente_anosov} in this subsection we extend Walters' stability theorem \cite{Wa78}*{Theorem 4}, which states that Anosov homeomorphisms are \emph{topologically stable} \cite{Wa78}*{Definition 5}, proving a similar result valid for the wider class of extensions of Anosov homeomorphisms. Note that the proof is based in the \mbox{$1\Rightarrow2$} implication of Theorem \ref{teo:cociente_anosov}.

\begin{theorem}\label{teo:estabilidad} 
 Suppose that $f$ is an extension of an Anosov homeomorphism $f_R\colon M_R\to M_R$ and let $q\colon M\to M_R$ be the canonical map. Then, given a neighborhood $\NN_q$ of $q$ in $\cont(M,M_R)$ there exists a neighborhood $\NN_f$ of $f$ in $\homeo(M)$ such that for all $g\in\NN_f$ there exists $q_g\in\NN_q$ such that $f_R\circ q_g=q_g\circ g$.

 Moreover, there exists a neighborhood $\NN_q^0$ of $q$ in $\cont(M,M_R)$ such that if the above $\NN_q$ is chosen so that $\NN_q\subseteq\NN_q^0$ then the map $q_g$ is uniquely determined by $g\in\NN_f$.
\end{theorem}

\begin{proof}
 Let $d_R$ a compatible metric on $M_R$ and $\alpha_R>0$ an expansivity constant for $f_R$ relative to $d_R$. By Theorem \ref{teo:cociente_anosov} there exists a compatible metric $d$ on $M$ and $\alpha>0$ such that $f$ is an $\alpha$-semi Anosov homeomorphism and $R=R(d,\alpha)$. Moreover, by Remark \ref{obs:anosovgordo} we can assume
 \begin{equation}\label{ec:a0}
  \alpha\leq\alpha_R\qquad\text{and}\qquad d_R([x],[y])\leq d(x,y)\quad\text{if}\quad x,y\in M.
 \end{equation}

 Let $0<\delta\leq\nicefrac\alpha4$ such that $f$ has the $\nicefrac\alpha4-\delta$ shadowing property and consider the neighborhood $\NN_f=\{g\in\homeo(M):d_{C^0}(f,g)<\delta\}$. Given $g\in\NN_f$ for all $x\in M$ the orbit $\xi=(g^nx)_{n\in\Z}$ of $x$ under $g$ is a $\delta$-pseudo orbit of $f$. Then there exists $y_x\in M$ that $\nicefrac\alpha4$-shadows $\xi$. If $y_x'\in M$ also $\nicefrac\alpha4$-shadows $\xi$ then $d(f^ny_x,f^ny_x')<\nicefrac\alpha2$ for all $n\in\Z$, hence $y_x\sim y_x'$. Therefore, in this way each $x\in M$ determines a unique class $[y_x]\in M_R$, that is, for each $g\in\NN_f$ we have a map $q_g\colon M\to M_R$ defined as $q_g(x)=[y_x]$ if $x\in M$, where $y_x\in M$ is any point $\nicefrac\alpha4$-shadowing the $g$-orbit of $x$.

 For $g\in\NN_f$ and $x\in M$ it is easy to see that the $g$-orbit of $gx$ is $\nicefrac\alpha4$-shadowed by $fy_x$, then $q_g(gx)=[fy_x]=f_R[y_x]=f_Rq_g(x)$, that is, $f_R\circ q_g=q_g\circ g$.

 To prove that $q_g$ is continuous let $x\in M$ and suppose $\varepsilon>0$ is given. As the quotient homeomorphism $f_R$ is uniformly expansive there exists $N\in\N$ such that
 \begin{equation}\label{ec:a1}
  \text{if }u,v\in M\text{ and }d_R(f_R^n[u],f_R^n[v])<\alpha_R\text{ for all }|n|\leq N\text{ then }d_R([u],[v])<\varepsilon.
 \end{equation}
 On the other hand, as $q\circ g^n$, $|n|\leq N$, are continuous there exists $\rho>0$ such that
 \begin{equation}\label{ec:a2}
  \text{if }z\in M\text{ and }d(x,z)<\rho\text{ then }d_R\bigl([g^nx],[g^nz]\bigr)<\nicefrac{\alpha_R}2\text{ for all }|n|\leq N.
 \end{equation}
 Hence, for all $z\in M$ such that $d(x,z)<\rho$, and for $|n|\leq\N$ we can estimate
 \begin{equation*}
 \begin{split}
  d_R\bigl(f_R^nq_g(x),f_R^nq_g(z)\bigr)
  &\leq d_R\bigl(f_R^n[y_x],[g^nx]\bigr)+d_R\bigl([g^nx],[g^nz]\bigr)+d_R\bigl([g^nz],f_R^n[y_z]\bigr)\\
  &=d_R\bigl([f^ny_x],[g^nx]\bigr)+d_R\bigl([g^nx],[g^nz]\bigr)+d_R\bigl([g^nz],[f^ny_z]\bigr)\\
  &\leq d\bigl(f^ny_x,g^nx\bigr)+d_R\bigl([g^nx],[g^nz]\bigr)+d\bigl(g^nz,f^ny_z\bigr)\\
  &<\nicefrac\alpha4+\nicefrac{\alpha_R}2+\nicefrac\alpha4\leq\alpha_R,
 \end{split}
 \end{equation*}
 where we used the conditions (\ref{ec:a0}) and (\ref{ec:a2}). Then, by condition (\ref{ec:a1}) we deduce that $d_R\bigl(q_g(x),q_g(z)\bigr)<\varepsilon$ if $x,z\in M$ and $d(x,z)<\rho$, proving the continuity of $q_g$.

 To see that $q_g\in\NN_q$ pick $\varepsilon>0$ such that $B_\varepsilon(q)\subseteq\NN_q$, and let $N\in\N$ as in condition (\ref{ec:a1}). Taking a smaller neighborhood $\NN_f$ if necessary we can suppose that $d_{C^0}(f^n,g^n)<\nicefrac{\alpha_R}2$ for all $|n|\leq N$ if $g\in\NN_f$. Then, for $g\in\NN_f$, $x\in M$ and $|n|\leq N$
 \begin{equation*}
 \begin{split}
  d_R\bigl(f_R^nq_g(x),f_R^nq(x)\bigr)
  &=d_R\bigl(f_R^n[y_x],f_R^n[x]\bigr)\\
  &=d_R\bigl([f^ny_x],[f^nx]\bigr)\\
  &\leq d\bigl(f^ny_x,f^nx\bigr)\\
  &\leq d(f^ny_x,g^nx)+d(g^nx,f^nx)\\
  &<\nicefrac\alpha4+\nicefrac{\alpha_R}2<\alpha_R.
 \end{split}
 \end{equation*}
 Hence, by condition (\ref{ec:a1}) $d_R\bigl(q_g(x),q(x)\bigr)<\varepsilon$ for all $x\in M$, that is, $q_g\in B_\varepsilon(q)\subseteq\NN_q$.

 For the last statement it is enough to take $\NN_q^0=B_{\alpha_R/2}(q)$: If $q_g,q_g'\in\NN_q^0$ verifies $f_R\circ q_g=q_g\circ g$ and $f_R\circ q_g'=q_g'\circ g$, then for all $x\in M$ and $n\in\Z$ we have
 $$
 d_R\bigl(f_R^nq_g(x),f_R^nq_g'(x)\bigr)=d_R\bigl(q_g(g^nx),q_g'(g^nx)\bigr)<\alpha_R.
 $$
 Hence $q_g(x)=q_g'(x)$ for all $x\in M$ because $\alpha_R$ is an expansivity constant for $f_R$.
\end{proof}

\begin{remark}
 If in the preceding Theorem \ref{teo:estabilidad} we consider the case in which $f=f_R$ is an Anosov homeomorphism and $q$ is the identity map of $M$, then we recover Walters' stability theorem \cite{Wa78}*{Theorem 4} in this special case. 
\end{remark}

\subsection{Shadowing envelopes}\label{subsec:envovente}

In this subsection we give another application of our study of extensions of Anosov systems, showing in Theorem \ref{teo:envolvente} that, loosely speaking, any expansive system can be enlarged to an expansive system satisfying the shadowing property for the pseudo orbits of the original system (see Definition \ref{def:envelope}). The proof is based on Proposition \ref{prop:sombreado_cociente} which was used to prove the \mbox{$2\Rightarrow1$} implication of Theorem \ref{teo:cociente_anosov}.

\medskip

Consider the space $\Sigma=M^Z$ of bi--sequences in $M$ with the product topology. This is a compact metrizable space. Concretely, we will use the compatible metric
$$\textstyle
d_\Sigma(\xi,\eta)=\sum_{k\in\Z}\frac{d(x_k,y_k)}{2^{|k|}},
$$
where $\xi=(x_k)_{k\in\Z}$ and $\eta=(y_k)_{k\in\Z}$ are elements of $\Sigma$. Let  $\sigma\colon\Sigma\to\Sigma$ be the \emph{shift} homeomorphism $\sigma(x_k)_{k\in\Z}=(x_{k+1})_{k\in\Z}$ if $(x_k)_{k\in\Z}\in\Sigma$, and $\iota\colon M\to\Sigma$ the embedding $\iota(x)=(f^nx)_{n\in\Z}$ if $x\in M$. It is easy to see that we have $\sigma\circ\iota=\iota\circ f$, that is, there is a copy of the system $(M,f)$ inside the system $(\Sigma,\sigma)$.

\begin{lemma}\label{lem:sombreado_shift}
 The dynamical system $(\Sigma,\sigma)$ has the shadowing property on the set $\Set=(\im\iota)^\Z$. Moreover, for every $\varepsilon>0$ there exists $\delta>0$ such that every $\delta$-pseudo orbit $\Xi=\bigl(\iota(x_n)\bigr)_{n\in\Z}$ belonging to $\Set$, where $x_n\in M$ for all $n\in\Z$, is $\varepsilon$-shadowed precisely by the point $\xi=(x_k)_{k\in\Z}\in\Sigma$.
\end{lemma}

\begin{proof}
 Given $\varepsilon>0$ let $\rho>0$ and $N\in\N$ such that if $\eta=(y_k)_{k\in\Z}$ and $\zeta=(z_k)_{k\in\Z}$
 \begin{equation}\label{ec:x01}
 d(y_k,z_k)<\rho\text{ for all }|k|\leq N\text{ implies } d_\Sigma(\eta,\zeta)<\varepsilon.
 \end{equation}
 For these $\rho>0$ and $N\in\N$ let $\delta>0$ such that for $\xi=(x_n)_{n\in\Z}$ we have
 \begin{equation}\label{ec:x02}
  \text{if }\xi\text{ is a }\delta\text{-pseudo orbit of }f\text{ then }d(f^kx_n,x_{n+k})<\rho\text{ for all }|k|\leq N,n\in\Z.
 \end{equation}
 Suppose that $\Xi=(\xi_n)_{n\in\Z}$ is a $\delta$-pseudo orbit of $\sigma$ that belongs to $\Set$. Then, for each $n\in\Z$ we have $\xi_n=\iota(x_n)$ for some $x_n\in M$, that is, $\xi_n=(f^kx_n)_{k\in\Z}$. As $\Xi$ is a $\delta$-pseudo orbit of $\sigma$ for all $n\in\Z$ we can estimate
 \begin{equation*}
 \begin{split}
  d(fx_n,x_{n+1})
  &\leq d_\Sigma\bigl((f^{k+1}x_n)_{k\in\Z},(f^kx_{n+1})_{k\in\Z}\bigr)\\
  &=d_\Sigma\bigl(\sigma(f^kx_n)_{k\in\Z},(f^kx_{n+1})_{k\in\Z}\bigr)\\
  &=d_\Sigma(\sigma\xi_n,\xi_{n+1})<\delta.
 \end{split}
 \end{equation*}
 Therefore the bi--sequence $\xi=(x_n)_{n\in\Z}$ is a $\delta$-pseudo orbit of $f$. Hence, applying condition (\ref{ec:x02}) and then condition (\ref{ec:x01}) with $\eta=\xi_n$ and $\zeta=\sigma^n\xi$ we deduce that
 $$
 d_\Sigma(\xi_n,\sigma^n\xi)=d_\Sigma\bigl((f^kx_n)_{k\in\Z},(x_{n+k})_{k\in\Z}\bigr)<\varepsilon
 $$
 for all $n\in\Z$. This proves that $\Xi$ is $\varepsilon$-shadowed by $\xi$ and we are done.
\end{proof}

\begin{definition}\label{def:envelope}
 Let $(M',d')$ a compact metric space and $f'\colon M'\to M'$ a homeomorphism. We say that the system $(M',f')$ is a \emph{shadowing envelope} of $(M,f)$ iff there exists an embedding (a continuous injective map) $\mu\colon M\to M'$ such that $f'\circ\mu=\mu\circ f$, and for every $\varepsilon>0$ there exists $\delta>0$ such that for every $\delta$-pseudo orbit $(x_n)_{n\in\Z}$ of $f$ there exists $x'\in M'$ such that $d'\bigl(f'^nx',\mu(x_n)\bigr)<\varepsilon$ for all $n\in\Z$.
\end{definition}

Note that if $(M',f')$ is a shadowing envelope of $(M,f)$ then replacing the metric $d$ of $M$ by the (equivalent) \emph{pull back} metric $d_\mu$, given by $d_\mu(x,y)=d'\bigl(\mu(x),\mu(y)\bigr)$ if $x,y\in M$, we can identify $(M,d_\mu)$ with the subspace $\im\mu\subseteq M'$ and think $M\subseteq M'$ and $f=f'|_M$. In this context the shadowing condition in Definition \ref{def:envelope} can be restated saying that $(M',f')$ has the shadowing property for the pseudo orbits of $M$, that is on the set $\Set=M^\Z$: for every $\varepsilon>0$ there exists $\delta>0$ such that every $\delta$-pseudo orbit of $M$ is $\varepsilon$-shadowed by some point $x'\in M'$.

By Lemma \ref{lem:sombreado_shift} we see that $(\Sigma,\sigma)$ is a shadowing envelope for every $f$, the map $\iota$ playing the role of the embedding $\mu$ of Definition \ref{def:envelope}. We will say more in Theorem \ref{teo:envolvente} if $f$ is expansive.

\medskip

Let us recall the following result from \cite{Ach} rephrasing expansiveness.

\begin{proposition}[\cite{Ach}*{Proposition 3.1}]\label{prop:po_semiexp}
 Let $\alpha>0$. The following conditions are equivalent.
 \begin{enumerate}
  \item[1.] $f$ is expansive with expansivity constant $\alpha$.
  \item[2.] For every $\varepsilon>0$ there exists $\delta>0$ such that
  $$
  \qquad\text{if}\quad d(x_n,y_n)\leq\alpha\text{ for all }n\in\Z\quad\text{then}\quad d(x_n,y_n)<\varepsilon\text{ for all }n\in\Z,
  $$
  for every pair of $\delta$-pseudo orbits $(x_n)_{n\in\Z}$ and $(y_n)_{n\in\Z}$ of $f$.
 \end{enumerate}
\end{proposition}

In the next Lemma \ref{lem:expgordo_shift} we show a nice connection between expansiveness and $[\varepsilon,\alpha]$-expansiveness, providing a family of examples of $[\varepsilon,\alpha]$-expansive systems.

\medskip

Given $\rho>0$ let $\Sigma^\rho=\{(x_k)_{k\in\Z}\in M^\Z:d(fx_k,x_{k+1})\leq\rho\text{ for all }k\in\Z\}$. This is a compact $\sigma$-invariant subset of $\Sigma$ that contains the copy $\im\iota$ of $M$.

\begin{lemma}\label{lem:expgordo_shift}
 If $f$ is expansive with expansivity constant $\alpha>0$ then for each $\varepsilon>0$ there exists $\rho>0$ such that the system $(\Sigma^\rho,\sigma)$ is $[\varepsilon,\alpha]$-expansive.
\end{lemma}

\begin{proof}
 Let $\xi=(x_k)_{k\in\Z}$ and $\eta=(y_k)_{k\in\Z}$ denote generic elements of $\Sigma$. Given $\varepsilon>0$ there exists $\varepsilon_1>0$ such that if $d(x_k,y_k)<\varepsilon_1$ for all $k\in\Z$ then $d_\Sigma(\xi,\eta)<\varepsilon$. By Proposition \ref{prop:po_semiexp} there exists $\rho>0$ such that if $\xi,\eta\in\Sigma^\rho$ and $d(x_k,y_k)\leq\alpha$ for all $k\in\Z$ then $d(x_k,y_k)<\varepsilon_1$ for all $k\in\Z$. This $\rho>0$ works because if $\xi,\eta\in\Sigma^\rho$ and $d_\Sigma(\sigma^n\xi,\sigma^n\eta)\leq\alpha$ for all $n\in\Z$ then $d(x_n,y_n)\leq\alpha$ for all $n\in\Z$, hence $d(x_n,y_n)<\varepsilon_1$ for all $n\in\Z$, and therefore $d_\Sigma(\xi,\eta)<\varepsilon$.
\end{proof}

We are ready now to present the main result of this subsection.

\begin{theorem}\label{teo:envolvente}
 If $f$ is expansive then $f$ has an expansive shadowing envelope.
\end{theorem}

\begin{proof}
 Consider the embedding $\iota\colon M\to\Sigma$ and for a bi--sequence $\xi=(x_n)_{n\in\Z}$ of elements of $M$ denote the corresponding bi--sequence in $\Sigma$ as $\iota(\xi)=\bigl(\iota(x_n)\bigr)_{n\in\Z}$.

 Let $\alpha>0$ be an expansivity constant of $f$. By Lemma \ref{lem:sombreado_shift} there exists $\delta>0$, which can be assumed $\delta\leq\nicefrac\alpha4$, such that every $\delta$-pseudo orbit $\Xi$ of $\sigma$ belonging to $\Set=(\im\iota)^\Z$ is $\nicefrac\alpha4$-shadowable, and moreover if $\Xi=\iota(\xi)$, where $\xi\in M^\Z$, then $\xi$ as an element of $\Sigma$ is a shadowing point for $\Xi$. For this $\delta>0$ by Lemma \ref{lem:expgordo_shift} there exists $\rho>0$ such that $(\Sigma^\rho,\sigma)$ is $[\nicefrac\delta4,\alpha]$-expansive.

 We claim that the system $(\Sigma^\rho,\sigma)$ has the $\nicefrac\alpha4-\delta$ shadowing property on the set $\Set^\rho=\iota(\Sigma^\rho)=\{\iota(\xi):\xi\in\Sigma^\rho\}$. Indeed, if $\xi\in\Sigma^\rho$ and $\iota(\xi)$ is a $\delta$-pseudo orbit of $\sigma$, then as we pointed out above $\iota(\xi)$ is $\nicefrac\alpha4$-shadowed by $\xi$ which belongs to $\Sigma^\rho$.

 Now consider the equivalence relation $R=R(d_\Sigma,\alpha)$ on the space $\Sigma^\rho$ as in Definition \ref{def:rel_quiv_de_Lew}. Denote $\Sigma^\rho_R=\Sigma^\rho/R$ the quotient space, $q\colon\Sigma^\rho\to\Sigma^\rho_R$ the canonical map, and $\sigma_R\colon\Sigma^\rho_R\to\Sigma^\rho_R$ the induced homeomorphism. By Theorem \ref{teo:cociente_expnsivo} we know that $\sigma_R$ is expansive, and by Proposition \ref{prop:sombreado_cociente} we have that $\sigma_R$ has the shadowing property on $q(\Set^\rho)$.

 We affirm that $(\Sigma^\rho_R,\sigma_R)$ is a shadowing envelope of $(M,f)$ with embedding map $\mu\colon M\to\Sigma^\rho_R$ given by $\mu=q\circ\iota$ (see Definition \ref{def:envelope}). Certainly $\mu$ is continuous, and $\sigma_R\circ\mu=\mu\circ f$ because $\sigma\circ\iota=\iota\circ f$ and $\sigma_R\circ q=q\circ\sigma$. To prove injectivity suppose $x,y\in M$ and $\mu(x)=\mu(y)$, that is $\iota(x)\sim\iota(y)$, where $\sim$ denotes $R$-equivalence. Then, as $\iota(z)=(f^kz)_{k\in\Z}$ if $z\in M$, for all $n\in\Z$ we have
 $$
 d(f^nx,f^ny)\leq d_\Sigma\bigl((f^{n+k}x)_{k\in\Z},(f^{n+k}y)_{k\in\Z}\bigr)=d_\Sigma\bigl(\sigma^n\iota(x).\sigma^n\iota(y)\bigr)\leq\alpha.
 $$
 Hence $x=y$ because $\alpha$ is an expansivity constant for $f$. Thus $\mu$ is injective.

 Finally, to prove the shadowing condition of Definition \ref{def:envelope} let $\varepsilon>0$ be given. As $\sigma_R$ has the shadowing property on $q(\Set^\rho)$ there exists $\delta'>0$ such that every $\delta'$-pseudo orbit of $\sigma_R$ belonging to $q(\Set^\rho)$ is $\varepsilon$-shadowable. By the continuity of $\mu$ there exists $\delta''>0$ such that if $x,y\in M$ and $d(x,y)<\delta''$ then $d_R\bigl(\mu(x),\mu(y)\bigr)<\delta'$. Let $\delta=\min\{\delta'',\rho\}$. If $\xi$ is a $\delta$-pseudo orbit of $f$ then on one hand $\xi\in\Sigma^\rho$ because $\delta\leq\rho$, so that $\mu(\xi)\in q(\Set^\rho)$. On the other hand, as $\delta\leq\delta''$ we have that $\mu(\xi)$ is a $\delta'$-pseudo orbit of $\sigma_R$. Hence $\mu(\xi)$ is $\varepsilon$-shadowable. This ends the proof.
\end{proof}

\section*{Acknowledgments}

We would thank to Alfonso Artigue and José Vieitez who have actively participated in this research. We also thank to the referee for the useful comments and suggestions on the original manuscript.

\medskip

Received xxxx 20xx; revised xxxx 20xx.

\medskip

\end{document}